\documentclass[11pt]{article}

\usepackage{graphicx}
\usepackage{pst-pdf}

\usepackage{amssymb}
\usepackage[intlimits]{amsmath}
\usepackage{amsfonts}
\usepackage{amsthm}

\usepackage{mathptmx}

\usepackage{hyperref}

\usepackage[a4paper]{geometry}

\let\oldsqrt\sqrt
\def\sqrt{\mathpalette\DHLhksqrt}
\def\DHLhksqrt#1#2{%
\setbox0=\hbox{$#1\oldsqrt{#2\,}$}\dimen0=\ht0
\advance\dimen0-0.2\ht0
\setbox2=\hbox{\vrule height\ht0 depth -\dimen0}%
{\box0\lower0.4pt\box2}}

\allowdisplaybreaks

\newcommand{\R}{\mathbb{R}} 
\newcommand{\dist}{\textnormal{dist}} 
\newcommand{\diam}{\textnormal{diam}} 
\newcommand{\supp}{\textnormal{supp}} 
\newcommand{\essinf}{\textnormal{essinf}} 
\newcommand{\hl}{(-\Delta)^{s}} 

\newcommand{\Om }{\Omega }
\newcommand{\de }{\partial }

\newcommand{\ov}{\overline}

\renewcommand{\phi}{\varphi}

\newcommand{\cB}{{\mathcal B}}

\newcommand{\cD}{{\mathcal D}}
\newcommand{\cE}{{\mathcal E}}

\newcommand{\cH}{{\mathcal H}}

\renewcommand{\a }{\alpha }
\renewcommand{\b }{\beta }
\renewcommand{\d}{\delta }

\newcommand{\e }{\varepsilon }

\newcommand{\n }{\nabla }
\renewcommand{\phi}{\varphi}

\renewcommand{\O }{\Omega }

\newcommand{\be}{\begin{equation}}
\newcommand{\ee}{\end{equation}}

\theoremstyle{definition}
\newtheorem{defi}{Definition}[section]
\newtheorem{bem}[defi]{Remark}

\theoremstyle{plain} 
\newtheorem{satz}[defi]{Theorem}
\newtheorem{prop}[defi]{Proposition}
\newtheorem{lemma}[defi]{Lemma}
\newtheorem{cor}[defi]{Corollary}

\theoremstyle{definition}

\numberwithin{equation}{section}
\title{Overdetermined problems with fractional Laplacian}
\author{Mouhamed Moustapha Fall \footnote{African Institute
for Mathematical Sciences (AIMS) of Senegal. Km 2, Route de Joal. BP 1418
Mbour, Senegal. mouhamed.m.fall@aims-senegal.org} \hspace{1ex}and
Sven Jarohs\footnote{Goethe-Universit\"{a}t Frankfurt, Institut
f\"{u}r Mathematik. Robert-Mayer-Str. 10 D-60054 Frankfurt, Germany.
jarohs@math.uni-frankfurt.de.} }
\date{\today}

\begin{document}
\maketitle
\begin{abstract}
Let $N\geq 1$ and $s\in (0,1)$.
In the present work we characterize bounded  open sets $\O$ with $ C^2$ boundary (\textit{not necessarily connected})
for which the following  overdetermined problem
\begin{equation*}
\left\{\begin{aligned}
( -\Delta)^s u &= f(u) &&\text{ in $\Omega$;}\\
 u&=0  &&\text{ in $\mathbb{R}^N\setminus \Omega$;}\\
 (\partial_{\eta})_s u&=Const.  &&\text{ on $\partial \Omega$}
\end{aligned}\right.
\end{equation*}
has a nonnegative and nontrivial solution, where $\eta $ is the outer unit normal vectorfield along
$\partial\Omega$
 and for $x_0\in\partial\Omega$
\[
 \left(\partial_{\eta}\right)_{s}u(x_{0})=-\lim_{t\to 0}\frac{u(x_{0}-t\eta(x_0))}{t^s}.
\]
 Under mild assumptions on $f$, we prove that $\Omega$ must be a
 ball. In the special case $f\equiv 1$, we obtain an extension of
 Serrin's result in 1971.  The fact that  $\Omega$ is not assumed to be connected is related to the nonlocal property of the fractional Laplacian.
   The main ingredients in our  proof  are maximum principles   and the method of moving planes.
\end{abstract}
{\footnotesize
\begin{center}
\end{center}
\begin{center}
\end{center}
}
\section{Introduction}
Let $\Omega\subset\R^N$, $N\geq2$, be a bounded connected open set with $C^2$ boundary.
 By the method of moving planes,  Serrin  proved, in his celebrated paper \cite{S71}, that if there is a solution of the overdetermined problem
\begin{equation}\label{eq:Serrin}
\left\{\begin{aligned}
 -\Delta u &= 1 &&\text{ in $\Omega$;}\\
 u&=0  &&\text{ on $\partial \Omega$;}\\
  u&>0  &&\text{ in $ \Omega$;}\\
 \partial_\eta u&=c  &&\text{ on $\partial \Omega$}
\end{aligned}\right.
\end{equation}

 then $\Omega$ must be a ball. The motivations for    this
 problem are both from fluid and solid mechanics.
 Serrin's argument initiated extensive researches in the field of overdetermined problems and symmetry properties of solutions to partial differential equations, see   e.g. \cite{GNN1,GNN2} for applications.
 The moving plane method  was first employed in geometry
 by Alexandrov in \cite{Alexandrov} to  characterize the sphere as the only
 embedded closed  hypersurfaces with constant mean curvature.
Over the years there are methods off moving plane to obtain Serrin's result, see e.g.   \cite{Wein-Serrin}, \cite{BNST} and also \cite{FK} which
 reveal  the application of Alexandrov's result in the study of
 \eqref{eq:Serrin}. In those arguments $\Omega$ were assumed to be connected.\\
 Several works have been devoted to related Serrin's problem.  Since a complete list of references cannot be given, we only quote:  \cite{AM, BD,  BH, BK, CS, CH, dLS,   FV1, FV2, FG, FGK, GL, HHP,  Pr, Ra, Re1, Re2, Si,  SS}.

   In the present
   work we  study an overdeterminded problem involving the fractional
    Laplacian $\hl$, $s\in(0,1)$. It is  defined for $u\in C^{\infty}_{c}(\R^{N})$ as
\[
 \hl u(x):= c_{N,s} P.V. \int_{\R^{N}}\frac{u(x)-u(y)}{|x-y|^{N+2s}}\ dy,
\]
where
\begin{equation}
 c_{N,s}=s(1-s)4^s\pi^{-N/2}\frac{\Gamma(\frac{N}{2}+s)}{\Gamma(2-s)}\label{fracsconstant}
\end{equation}
is a normalization constant (see e.g. \cite{NPV11}) so that
$\widehat{\hl u}(\xi)=|\xi|^{2s}\widehat{u}(\xi)$.
 Due to
applications in Physics, Biology and Finance, differential
equations involving the fractional Laplacian $\hl$ have received
growing attention in recent years (see e.g.
\cite[Introduction]{NPV11} and the references therein) but still much less understood than their non-fractional
counterparts.   The
overdeterminded problem we are interested in is
$$
(\ast)\left\{\begin{aligned}
 \hl u &= 1 &&\text{ in $\Omega$;}\\
 u&=0  &&\text{ in $\R^{N}\setminus\Omega$;}\\
 \left(\partial_\eta\right)_{s} u&\equiv c  &&\text{ on $\partial \Omega$,}
\end{aligned}\right.
$$
where $c $ is a negative constant, $\Omega\subset\R^{N}$ is a bounded open set with $C^{2}$ boundary, $\eta(x_0)$ stands for the outer normal at $x_{0}\in\partial \Omega$ and

\[
 \left(\partial_{\eta}\right)_{s}u(x_{0}):=-\lim_{t\to 0}\frac{ u(x_{0}-t\eta(x_0))}{t^s}.
\]
We note that solutions to the first two equations in
$(\ast)$ are  in $C^s(\overline{\Omega})$ (see e.g. \cite[Proposition 2.9]{S05} or \cite[Proposition 1]{PSV13}) and
therefore the overdetermined Neuman condition make sense pointwise. Our first result is contained in the following
\begin{satz}\label{main}
 Let $\Omega\subset\R^{N}$, $N\geq 1$, be a bounded open set with $C^{2}$ boundary. Assume
 that there is a   solution $u\in C^s(\overline{\Omega})$ of $(\ast)$ satisfying
 \begin{equation}\label{missing assumption}
 	\frac{u}{\delta^s_{\Omega}}\in C^1(\overline{\Omega}),\quad\text{with}\quad\delta_{\Omega}(x):=\dist(x,\R^N\setminus \Omega).
 \end{equation}
 Then $\Omega=B_R(x_0)$, a ball centered at $x_0$ with radius $R$. In addition  $u(x)=\gamma_{N,s}\left(R^2-|x-x_{0}|^2\right)^s$
 for all $x\in B_R(x_0)$, where
 \[
\gamma_{N,s}:=\frac{4^{-s}\Gamma(\frac{N}{2})}{\Gamma(\frac{N}{2}+s)\Gamma(1+s)}.
\]
\end{satz}

We point out that the precise nonnegative and nontrivial  function solving $(\ast)$ when $\Omega$ is a ball were
calculated  in  \cite{BB00} (see also \cite{D12}).
 We should mention that   the overdetermined problem
$(\ast)$ arises from  Euler-Lagrange equations of
domain dependent variational problems involving the fractional
Laplacian. We refer to \cite{DGV12} where the authors considered the
case $N=2$, $s=1/2$ and $\partial\Omega$ of class $C^\infty$ and assumed $\Omega$ to be connected. In
addition, under the  aforementioned assumptions,  the authors
obtained similar result via the moving plane method.
 We also mention that the authors in \cite{DGV12} used the fact that the $1/2$-Laplacian can be
realized locally via a  Dirichlet-to-Neumann operator.

Here we use the moving plane argument only in the fractional framework which is a fundamental difference with respect to \cite{DGV12}. In particular our approach could be feasible for more general types of integro-differential equations. Moreover, we take advantages to the nonlocal
structure of the fractional laplacian to construct  subsolutions
yielding a Hopf lemma (see Lemma \ref{hopf}) and a version of
Serrin's corner boundary point lemma (see Lemma \ref{lem:max-corner}). The nonlocal
structure of $(-\Delta)^s$ forces the Dirichlet condition in $\R^N\setminus\Omega$.
This is a key property which allows us to prove that  $\Omega$ must  be connected at some stage during the moving plane process.\\
The moving plane method have been also employed within the fractional framework in \cite{BLW05, Chen_Li_Ou,  FW2, FQT, JW, FW13}.\\

Our next result is a generalization of the previous one.

\begin{satz}\label{main2}
  Let $c\in\R$  and  $\Omega\subset\R^{N}$, $N\geq 1$, be a bounded open set with  $C^{2}$ boundary.
   Furthermore, let $f:\R\to\R$ be locally Lipschitz and assume that there is a function $u\in C^s(\overline{\Omega})$, which is nonnegative, nontrivial and satisfies \eqref{missing assumption} and
\begin{equation}\label{new prob}
\left\{\begin{aligned}
 \hl u &= f(u) &&\text{ in $\Omega$;}\\
 u&=0  &&\text{ in $\R^{N}\setminus\Omega$;}\\
 \left(\partial_\eta \right)_{s} u&=c  &&\text{ on $\partial \Omega$.}
\end{aligned}\right.
\end{equation}
 Then $\Omega$ is a ball and $u>0$ in $\Omega$.
\end{satz}

\begin{bem}
	We point out that Assumption \eqref{missing assumption} has been missing in a previous version, but is used in a key step of the proof, see Lemma \ref{0-derivative} below. This has been noted in \cite{DPTV23}, \cite[Lemma and Remark 3.1]{JKS25} and it can be seen as the fractional counterpart to the assumption of $u$ being in $C^2(\overline{\Omega})$ in the classical result by Serrin for $s=1$. If, in addition, $\Omega$ is assumed to be of class $C^{2,\alpha}$ for some $\alpha>0$, then the solution of ($\ast$) is known to satisfy \eqref{missing assumption}. Moreover, if $s>\frac12$, then this is also true for a solution to \eqref{new prob}.
\end{bem}

The paper is organized as follows. In Section \ref{defs} we will
give some preliminaries and   notation.  Dealing with
  the moving plane method leads to dealing with antisymmetric functions so we will give a proof of Hopf's Lemma and some  maximum principles for antisymmetric functions in  Section
  \ref{anti-results}.
  Finally, Section \ref{overdetermined}
     is devoted to the proof of Theorem \ref{main} and in Section \ref{s:Gen}, we prove Theorem \ref{main2}.

\medskip
\noindent \textbf{Acknowledgement:} The authors would like to thank
Tobias Weth for helpful discussions and  Xavier Ros-Oton for taking
their attention to \cite{D12}. Part of this work was done while the
second author was visiting AIMS-Senegal. He would like to thank them
for their kind hospitalities. The first author is supported by the Alexander von Humboldt foundation.

\section{Definitions and Notation}\label{defs}
Let  $N\geq 1$ and  $s\in (0,1)$.
For $u,v\in H^{s}(\R^{N})$, we consider the bilinear form induced by the fractional laplacian:
\[
 \cE(u,v):=\frac{c_{N,s}}{2}\int_{\R^{N}}\int_{\R^{N}}\frac{(u(x)-u(y))(v(x)-v(y))}{|x-y|^{N+2s}}\ dxdy.
\]
 Furthermore, let
\[
 \cH^{s}_{0}(\Omega)=\{u\in H^{s}(\R^{N})\;:\; u=0 \text{ on $\R^{N}\setminus \Omega$}\},
\]
where $\Omega\subset \R^{N}$ is an arbitrary open set.
 If $\Omega$ is bounded, we define the first Dirichlet eigenvalue:
$$
\lambda_1(\Omega)=\min_{u\in \cH^{s}_{0}(\Omega)
}\frac{\displaystyle\cE(u,u)}{\displaystyle\int_{\Omega}u^2dx}.
$$
Then we have  $\lambda_1(\Omega)\geq C_{N,s}|\O|^{-\frac{2s}{N}}$, see e.g. \cite{YY}, where $C_{N,s}=\frac{N}{2s}|B_1(0)|^{1+2s/N}c_{N,s}$ and $c_{N,s}$ is the constant in (\ref{fracsconstant}).

In the following all equalities involving $\hl$ are meant in the weak sense, i.e. for $g\in L^{2}(\Omega)$
 we call $u\in \cD^{s}(\Omega):=\{u:\R^{N}\to \R \text{ measurable}\;:\; \cE(u,\varphi)<\infty\text{ for all }\varphi\in \cH^{s}_{0}(\Omega)\}$ a \textit{supersolution (subsolution)} of
\begin{equation}\label{eq:Dsueqg}
 \hl u= g \quad   \text{ in $\Omega$},
\end{equation}
if for all $\varphi \in \cH^{s}_{0}(\Omega)$, $\varphi \geq 0$
\[
 \cE(u,\varphi)\geq \int_{\Omega}g(x) \varphi(x)\ dx  \quad\left( \cE(u,\varphi)\leq \int_{\Omega}g(x) \varphi(x)\ dx\right).
\]
We call $u\in \cD^{s}(\Omega)$ an entire supersolution (subsolution) if additionally we have
\[
 u\geq0 \text{ on $\R^{N}\setminus \Omega$} \quad (  u\leq0 \text{ on $\R^{N}\setminus \Omega$}).
\]
The function $u$ is called a \textit{solution}, if $u$ is an entire supersolution and an entire subsolution.

We note that if $u$ is an entire supersolution of (\ref{eq:Dsueqg}) then $u_{-}=-\min\{w,0\}\in \cH^{s}_{0}(\Omega)$.

Finally we want to note that if $u\in C^{1,1}(\Omega) $,
for some open set $\Omega\subset\R^{N}$, we have that $\hl u$
is continuous on $\Omega$ (see e.g. \cite[Proposition 2.5]{S05}).
 Thus if $u$ has such regularity  \eqref{eq:Dsueqg} holds pointwise.\\

The following further notation is used throughout the paper:
 for $x \in \R^N$ and $r>0$, $B_r(x)$ is the open ball centered at $x$ with radius
$r$ and $\omega_{N}$ will denote the volume of the $N$-dimensional ball with radius $1$.
Moreover, we denote $S^{1}:=\{x\in\R^{N}\;:\; |x|=1\}$.
For any subset $M \subset \R^N$, we denote by $1_M: \R^N \to \R$ the
characteristic function of $M$ and $\diam(M)$ the diameter of
$M$. The notation $A\subset\subset B$, $A,B\subset \R^{N}$ means that we have $\overline{A}\subset B$ and $\overline{A}$ is compact and nonempty.

Moreover, $w_+= \max\{w,0\}$ and $w_-=-\min\{w,0\}$ denote the
positive  and negative part of $w$ resp. For any $M\subset\R^{N}$, $|M|$ denotes the Lebesgue measure of $M$ and for $D,U\subset\R^{N}$ we set
$$
\dist(D,U):= \inf(|x-y|\::\: x \in D,\, y \in U\}.
$$
If $D= \{x\}$, we simply write $\dist(x,U)$ in place
of $\dist(\{x\},U)$. Finally, we denote for $\Omega\subset\R^{N}$,
\[
 \delta(x):= \delta_{\Omega}(x):=\dist(x,\R^{N}\setminus \Omega)
\]
the distance to the complement of a set $\Omega$. We will omit the subindex $\Omega$, whenever no confusion is possible.
\\[0.1cm]

\section{Maximum principles for entire antisymmetric supersolutions}\label{anti-results}

Due to the fact that we will work with the moving plane method in the following,
 we will need to prove some results concerning antisymmetric functions.
  Let $H\subset \R^{N}$ be a halfspace, i.e. for any $\lambda\in \R$, $e\in S^{1}$
  we consider
$$
  H:=H_{\lambda,e}:=\{x\in \R^{N}\;:\; x\cdot e>\lambda\}.
$$
Let $\Omega\subset H$ be a  bounded open set. Let
$T:=\partial H$ and
   denote by $Q:\R^{N}\to \R^{N}$, $x\mapsto \bar{x}$ the reflection of $x$ at $T$,
   i.e. $\bar{x}=x-2(x\cdot e)e+2\lambda e$.  We will  call $u\in \cD^{s}(\Omega)$ an \textit{entire antisymmetric supersolution}
   of $\hl u=g$ in $\Omega$, if $u$ is a supersolution of $\hl u=g$ in $\Omega$ and if additionally we have $u\geq0$
    on $H\setminus \Omega$ and $u(\bar{x})\leq-u(x)$ for all $x\in H$.

\begin{prop}[Weak Maximum Principle]\label{4-elliptic-max}
Let $H$ be a halfspace  and let
$\Omega\subset H$ be any open, bounded set, let $c\in
L^{\infty}(\Omega)$ be such that $c\leq
c_{\infty}<\lambda_1(\Omega)$ in $\Omega$ for some $c_{\infty}\geq0$
and let $g\in L^{2}(\Omega)$ be, such that $g\geq -\kappa$ with
\[
 0\leq\kappa<\frac{\lambda_1(\Omega)-c_{\infty}}{|\Omega|^{1/2}}.
\]
If $u$ is an entire antisymmetric supersolution of
\begin{equation}\label{4-elli}
\begin{aligned}
 \hl u &= c(x)u+g(x) &&\text{ in $\Omega$}\\
\end{aligned}
\end{equation}
then $\|u_{-}\|_{L^{2}(\Omega)}\leq\kappa|\Omega|^{1/2}/(\lambda_1(\Omega)-c_{\infty})<1$.
In particular, if $\kappa=0$ then $u\geq 0$ almost everywhere in
$\Omega$.
\end{prop}
\begin{proof}
Note that $\varphi:=u_{-}1_{H}\in \cH^{s}_{0}(\Omega)$ and
\begin{align*}
 (u(x)-u(y))(\varphi(x)-\varphi(y))+(\varphi(x)-\varphi(y))^{2}&=-u(x)\varphi(y)-u(y)\varphi(x)-2\varphi(x)\varphi(y)\\
&=-\varphi(x)\left(\varphi(y) + u(y)\right)-\varphi(y)\left(\varphi(x) +u(x)\right).
\end{align*}
Thus we have
\begin{align}
 \cE(u,\varphi)&=-\cE(\varphi,\varphi)-c_{N,s}\int_{\R^{N}}\int_{\R^{N}}\frac{\varphi(y)\left(\varphi(x)+u(x)\right)}{|x-y|^{N+2s}}\ dxdy\notag\\
&=-\cE(\varphi,\varphi)-c_{N,s}\int_{H}\int_{H}\frac{\varphi(y)u_{+}(x)}{|x-y|^{N+2s}}+\frac{\varphi(y)u(\bar{x})}{|\bar{x}-y|^{N+2s}}\ dxdy\notag\\
&\leq-\cE(\varphi,\varphi)-c_{N,s}\int_{H}\int_{H}\frac{\varphi(y)u_{+}(x)}{|x-y|^{N+2s}}-\frac{\varphi(y)u(x)}{|\bar{x}-y|^{N+2s}}\ dxdy\notag\\
&=-\cE(\varphi,\varphi)-c_{N,s}\int_{H}\varphi(y)\int_{H}u_{+}(x)\left(\frac{1}{|x-y|^{N+2s}}-\frac{1}{|\bar{x}-y|^{N+2s}}\right) + \frac{u_{-}(x)}{|\bar{x}-y|^{N+2s}}\ dxdy\notag\\
&\leq -\cE(\varphi,\varphi)\label{3-ineq-max}.
\end{align}
Thus we have
\begin{align*}
0&\leq \cE(u,\varphi)+\int_{\Omega}c(x)\varphi^{2}(x)\ dx-\int_{\Omega}g(x)\varphi(x)\ dx\\
&\leq -\cE(\varphi,\varphi) + c_{\infty}\|\varphi\|^2_{L^{2}(\Omega)} + \kappa \int_{\Omega}\varphi(x)\ dx\\
&\leq
\left(c_{\infty}-\lambda_1(\Omega)\right)\|\varphi\|^2_{L^{2}(\Omega)}
+ \kappa |\Omega|^{1/2}\|\varphi\|_{L^{2}(\Omega)}.
\end{align*}
If $\|\varphi\|_{L^{2}(\Omega)}>1$  we have, since $\|\varphi\|_{L^{2}(\Omega)}< \|\varphi\|^2_{L^{2}(\Omega)}$,
 \begin{align*}
0&\leq \left(c_{\infty}-\lambda_1(\Omega)\right)\|\varphi\|^2_{L^{2}(\Omega)} + \kappa |\Omega|^{1/2}\|\varphi\|_{L^{2}(\Omega)}\\
&<
\left(c_{\infty}+\kappa|\Omega|^{1/2}-\lambda_1(\Omega)\right)\|\varphi\|_{L^{2}(\Omega)}\leq
0,
\end{align*}
since $c_{\infty}-\lambda_1(\Omega)+\kappa|\Omega|^{1/2}<0$,
resulting in a contradiction. Thus we must have
$\|\varphi\|_{L^{2}(\Omega)}\leq 1$. In this case we have
 \begin{align*}
0&\leq
\left(-\left(\lambda_1(\Omega)-c_{\infty}\right)\|\varphi\|_{L^{2}(\Omega)}
+ \kappa |\Omega|^{1/2}\right)\|\varphi\|_{L^{2}(\Omega)},
\end{align*}
Thus we must have
$\|\varphi\|_{L^{2}(\Omega)}\in[0,\kappa|\Omega|^{1/2}/\left(\lambda_1(\Omega)-c_{\infty}\right)]$,
finishing the proof.
\end{proof}

\begin{bem}\label{maxprinzip-zusatz}
 Note that the result also holds if $u$ is an entire supersolution of $\hl u= c(x) u+g$ in $\Omega$,
  where $\Omega\subset \R^{N}$ is an arbitrary open, bounded set. The proof then simplifies, since the inequality (\ref{3-ineq-max}) follows trivially.
\end{bem}

For the proof of the following version of Hopf's Lemma,  we will need the function $\psi_{B}$  satisfying
\begin{equation}\label{ball}
\begin{aligned}
 \hl \psi_{B} &= 1, &&\text{ in $B$;}\\
\end{aligned}
\end{equation}
where $B\subset\R^{N}$ is a  ball. If $B=B_{r}(x_{0})$ for some $r>0$ and $x_{0}\in \R^{N}$ we have (see e.g. \cite{BB00} or \cite{D12})
\[
\psi_{B}(x)= \gamma_{N,s}\left(r^{2}-|x-x_{0}|^2\right)_{+}^{s}, \quad \gamma_{N,s}:=\frac{4^{-s}\Gamma(\frac{N}{2})}{\Gamma(\frac{N}{2}+s)\Gamma(1+s)}.
\]
It  was proved  in \cite[Lemma 4.3]{BLW05}   that if $u$ is a continuous supersolution
 of $\hl u=0$ in some open set $\Omega\subset\R^{N}$, such that $u\equiv 0$ on $\R^{N}\setminus \Omega$,
 then we  have for any outernormal $\eta$ and any $x_{1}\in \partial\Omega$
 such that there is an interior ball $B\subset \Omega$ with $x_{1}\in \partial B\cap \partial \Omega$, that $\partial_{\eta} u(x_{1})=-\infty$
 this was enough to carry over a moving plane argument. Here  we will need precise behavior of antisymmetric solutions
 near some corner points at the boundary. A recent result in \cite{RS12} states that  if $u\in \cH^{s}_{0}(\Omega)$ is a
solution of $\hl u=g$ in $\Omega\subset\R^{N}$, with $g\in
L^{\infty}(\Omega)$, then $\frac{u}{\delta^s}\in
C^{0,\alpha}(\overline{\Omega})$, for some $\alpha\in (0,1)$. For $x_1\in \partial\Omega$, it therefore make sense to define
\[
 \left(\partial_{\eta}\right)_{s}u(x_{1}):=-\lim_{t\to 0^{+}} \frac{ u(x_{1}-t\eta(x_1))}{t^{s}}.
\]

\begin{prop}[Hopf's lemma]\label{hopf}
 Let $H\subset \R^{N}$ be a halfspace  and  $\Omega\subset H$. Consider a  ball $B_1\subset\subset H$ and  $B_1\subset\Omega$.
 Furthermore, let $c\in L^{\infty}(\Omega)$ and assume that $c_0=\|c\|_{L^{\infty}(\Omega) } <\lambda_1(B_1)$. Let  $u\in \cD^{s}(\Omega)$  be an entire antisymmetric supersolution of
\begin{equation}\label{4-elli-2}
\begin{aligned}
 \hl u &= c(x)u, &&\text{ in $\Omega$;}\\
\end{aligned}
\end{equation}
such that
\begin{equation}\label{4-elli-3}
\begin{aligned}
 u&\geq0,  &&\text{ in $H$.}\\
\end{aligned}
\end{equation}
Let $K$ be a measurable set with  $K\subset\subset H\setminus \overline{B_1}$, $|K|>0$ and  suppose that $\essinf_{K}u >0$. Then
there is a constant $d=d(N,s,\diam(B_1),K,\textrm{dist}(B_1,K),c_0,\essinf_{K}u)>0$ such that
\[
 u(x)\geq d \delta_{B_1}^{s}(x)\quad  \text{ for almost every $x\in B_1$.}
\]
In particular, if $u\in C(\overline{B_1})$ and $u(x_0)=0$, for some
$x_0\in\partial B_1$, then we have
\[
- \liminf_{t\to 0^{+}} \frac{u(x_{0}-t\eta(x_0))}{t^{s}}<0.
\]
\end{prop}
   \begin{center}
\begin{figure}[htb]
   \begin{center}
  \fbox{\begin{minipage}{8cm}
   \begin{center}
     \includegraphics[width=\textwidth]{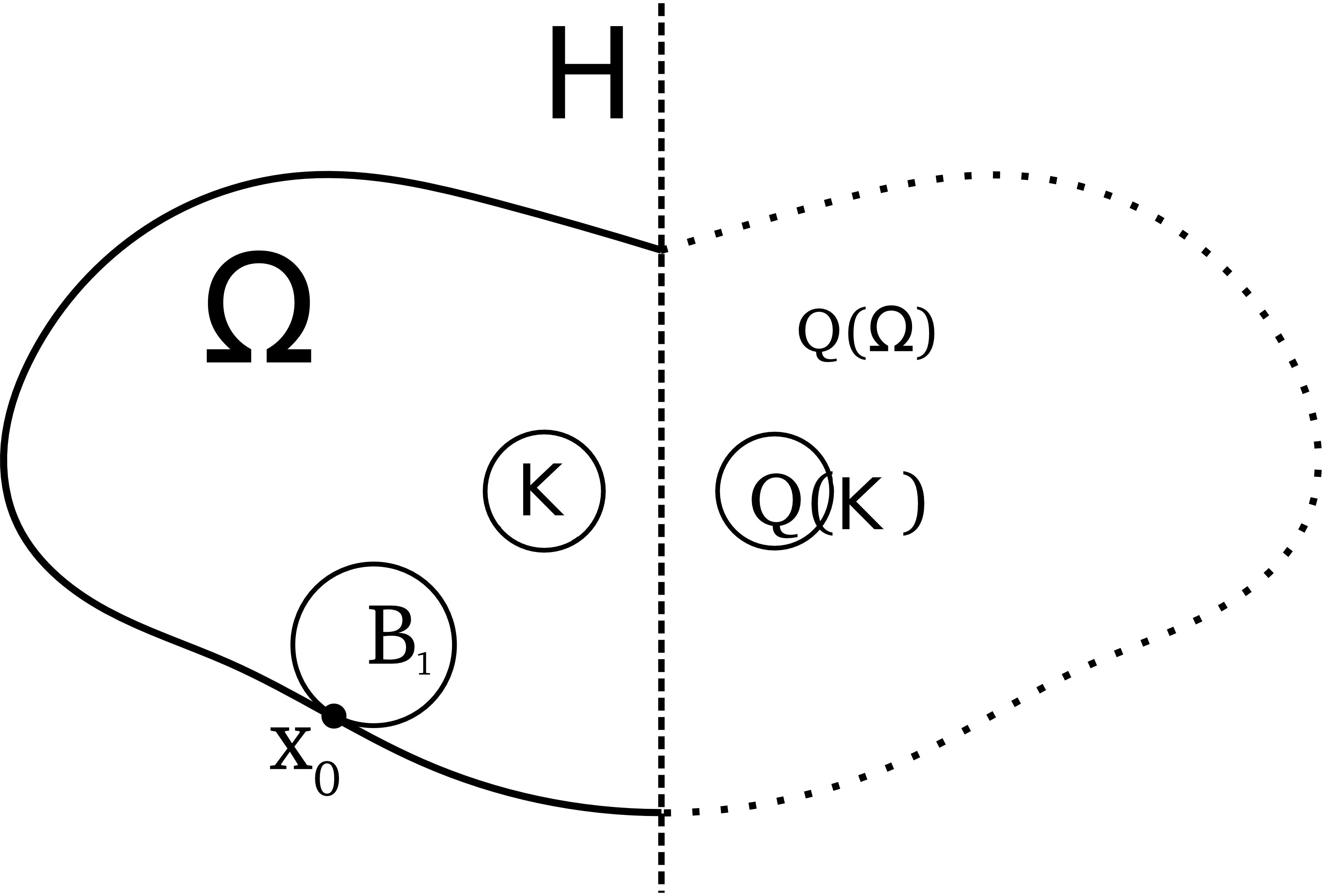}
   \end{center}
  \end{minipage}}
\caption{$K$ is an arbitratry measurable set of positive measure.}
   \end{center}
\end{figure}
   \end{center}
\begin{proof}
For $\alpha>0$, consider the barrier
\[
 w(x):=\psi_{B_{1}}(x)+\alpha1_{K}(x)-\psi_{Q(B_1)}(x)-\alpha 1_{Q(K)}(x),
\]
where $Q:\R^{N}\to\R^{N}$, $x\mapsto\bar{x}$ is the reflection at $\partial H$ and $1_M$ is the characteristic function for any $M\subset\R^N$. 
 Let $\varphi\in\cH^{s}_{0}(B_{1})$. Then we have 
\begin{align*}
 \cE(w,\varphi)&=\cE(\psi_{B_{1}},\varphi)+\alpha\cE(1_{K},\varphi)-\cE(\psi_{Q(B_{1})},\varphi)-\alpha\cE(1_{Q(K)},\varphi)\\
&=\int_{B_{1}}\varphi(x)\ dx -\alpha c_{N,s}\int_{B_{1}}\int_{K}\frac{\varphi(x)}{|x-y|^{N+2s}}\ dydx\\
&\qquad\qquad+\alpha c_{N,s}\int_{B_{1}}\int_{Q(K)}\frac{\varphi(x)}{|x-y|^{N+2s}}\ dydx+
c_{N,s}\int_{B_{1}}\int_{Q(B_{1})}\frac{\psi_{Q(B_{1})}(y)\varphi(x)}{|x-y|^{N+2s}}\ dydx.
\end{align*}
Therefore $w\in \cD^{s}(B_1)$ since $\dist(B_1,K)>0$. If moreover  $\varphi\geq0$, 
\begin{align*}
 \cE(w,\varphi)
&\leq \int_{B_{1}}\varphi(x) \left(\kappa-\alpha c_{N,s}\int_{K}\left(|x-y|^{-N-2s}-|x-\bar{y}|^{-N-2s}\right)\ dy\right)dx,
\end{align*}
where $\kappa:=1+\sup_{x\in B_1}\psi_{B_1}(x)\cdot \sup_{x\in B_1, \bar{y}\in H}|x-y|^{-N-2s}<\infty$ since $\dist(B_1,\R^{N}\setminus H)>0$. Moreover since $K$ has a positive distance to $\partial H$ and $B_1$ we have
\[
C_1=C_1(N,s,K,B_1)=c_{N,s}|K| \inf_{x\in B_{1},y\in K}|x-y|^{-N-2s}-|x-\bar{y}|^{-N-2s}>0.
\]
With this we have
\begin{equation*}
\cE(w,\varphi)\leq \int_{B_{1}}\varphi(x) \left(\kappa-\alpha C_{1}\right)dx.
\end{equation*}
Thus we may take $\alpha$ large so that $\kappa-\alpha C_1\leq -c_{0}\sup_{x\in B_{1}}\psi_{B_{1}}(x)$ and thus we have
\[
 \hl w\leq c(x)w \text{ in $B_{1}$.}
\]
 Note furthermore that, by construction,  $w= 0$ on $\R^{N}\setminus (B_{1}\cup K)$  and
 $w(\bar{x})= -w(x)$ for all $x\in \R^{N}$. By assumption, we can pick
 $$
 \e=\frac{\essinf_{K}u}{\alpha}>0
 $$
 so that  $v(x):=u(x)-\epsilon w(x) \geq0$ on $H\setminus B_{1}$.
  Since we chose $B_1$  such that $\|c\|_{L^{\infty}(B_1)}<\lambda_1(B_1)$, we can apply
  Proposition \ref{4-elliptic-max} to the supersolution $v$ yielding  $u\geq \epsilon w= \epsilon \psi_{B_{1}}$ in $B_{1}$.
  Finally, if   $u(x_{0})=0$ with $x_0\in\partial B_1$,  we have
\begin{align*}
  - \liminf_{t\to 0^{+}} \frac{u(x_{0}-t\eta(x_0))}{t^{s}}
&\leq -\epsilon \lim_{t\to 0^{+}}
\frac{\psi_{B_{1}}(x_{0}-t\eta(x_0))}{t^{s}}<0.
\end{align*}

\end{proof}

As a consequence of Proposition \ref{hopf}, we have

\begin{cor}[Strong maximum principle]\label{strong1}
 Let $H\subset \R^{N}$ be a halfspace and let $\Omega\subset H$ be an open bounded set.
 Furthermore let $c\in L^{\infty}(\Omega)$ and $u\in \cD^{s}(\Omega)$ be an entire antisymmetric supersolution of
\begin{equation}\label{4-elli-4}
\begin{aligned}
 \hl u &= c(x)u, &&\text{ in $\Omega$.}\\
\end{aligned}
\end{equation}
If $u\geq0$ in $H$   then either $u\equiv0$ in $H$ or $u>0$ in
$\Omega$.
\end{cor}

\begin{bem}\label{bemhopf}
We emphasize  that Proposition \ref{hopf} remains true for entire
supersolutions. Indeed, one would repeat the same proof by
considering the barrier $
w(x)=\psi_{B_{1}}(x)+\alpha\psi_K(x)$.  In particular
Corollary \ref{strong1} is also  valid for entire supersolutions.
  \end{bem}

\section{Proof of Theorem \ref{main}}\label{overdetermined}
Our objective in this section is to prove Theorem \ref{main} for the case $N\geq2$.  We present an argument based on the fact that $u >0$ in $\Omega$. As we shall see, simple observations shows that after two steps of the moving plane argument, we see that $\Omega$ must be connected.   For the case $N\geq 1$ and general right hand side, we postpone the proof in the section.
\begin{satz}[Theorem \ref{main} for $N\geq 2$]\label{teil1}
 Let $\Omega\subset\R^{N}$, $N\geq 2$, be an open, bounded set such that $\partial \Omega$ is $C^{2}$
 and assume that there is a solution $u\in   C^{s}(\overline{\Omega})$ of
\[
 \hl u =1 \quad\text{ in $\Omega$,}\qquad
  u=0 \quad\text{ in $\R^N\setminus\Omega$.}
\]
If there is a negative real number $c$ such that
\[
 \left(\partial_{\eta}\right)_{s} u\equiv c \quad \text{ on $\partial \Omega$},
\]
then $\Omega$ is a ball and $u=\psi_{\Omega}$, where $\psi_{\Omega}$ is given as in Section \ref{anti-results}.
\end{satz}

\begin{bem}\label{strongproperties}
We note that by regularity theory    $u\in C^\infty(\O)$.
In addition a nontrivial solution  $u\in \cH^{s}_{0} $ to  $\hl u
=1 \quad\text{in $\Omega$}$ is strictly positive in $\Omega$ by   Hopf's lemma  (see also Remark \ref{bemhopf}) and there must be $c <0$.
\end{bem}

\begin{proof}[Proof of Theorem \ref{teil1}]
  Let $e\in S^{1}$ be fixed and consider $T_{\lambda}:=T_{e,\lambda}:=\{x\in\R^{N}\;:\; x\cdot e=\lambda\}$
  as a hyperplane in $\R^{N}$, which we will continuously move by continuously varying $\lambda$.
  Since $\Omega$ is bounded, denote $l:=\max_{x\in \Omega}x\cdot e$, so that $T_{\lambda}\cap \Omega=\emptyset$ for $\lambda\geq l$.
   Denote $H_{\lambda}:=H_{e,\lambda}:=\{x\in\R^{N}\;:\; x\cdot e>\lambda\}$ and define $\Omega_{\lambda}:=\Omega_{e,\lambda}:=\Omega\cap H_{\lambda}$.
    Let $Q_{\lambda}:=Q_{e,\lambda}$ be the reflection about $T_{\lambda}$ as described in Section \ref{anti-results}
     and denote $\Omega'_{\lambda}:=Q_{\lambda}(\Omega_{\lambda})$, i.e. the reflection of $\Omega_{\lambda}$ about $T_{\lambda}$.
      Since $\partial\Omega$ is $C^{2}$ we have that for $\lambda<l$ but close to $l$,
that $\Omega'_{\lambda}\subset\Omega$. As we decrease $\lambda$, i.e. continue moving $T_{\lambda}$, two possible situations may occur:
\begin{equation}\label{ST12}
\begin{array}{ll}
\textrm{Situation 1:  There is a point $P_{0}\in \partial\Omega\cap \overline{\Omega'_{\lambda}}\setminus T_{\lambda}$ or}\\
\textrm{Situation 2: $T_{\lambda}$ is orthogonal to $\partial\Omega$ at some point $P_0\in \partial\Omega\cap T_{\lambda}$.}\\
\end{array}
\end{equation}

We note, that although $\Omega$ is not necessarily connected, there is no other possibility since $\partial\Omega$ is $C^2$ and $\Omega$ is bounded.
\begin{equation}\label{eq:defl0}
\textrm{ Let $\lambda_{0}$ be the point at which one of these
situations occur for the first time.}
\end{equation}
For simplicity, we put $T=T_{\lambda_{0}}$ and
$H=H_{\lambda_{0}}$. Our aim is to prove that if any of the above
situation occurs, $\Omega$ must be symmetric with respect to the
plane $T$.\\

To prove that the situations yield symmetry, we let $Q$ be the
reflection about $T$ as described in Section \ref{anti-results}. Then
define $Q(x)=:\bar{x}$ and consider the function
\[
 v(x):=u(x)-u(\bar{x})\quad\text{ for $x\in\R^{N}$.}
\]
Since $U:=\Omega'_{\lambda_{0}}\subset\Omega$ we have that $v$ satisfies
\[
\hl v=0 \text{ in $U$}
\]
and
\[
 \begin{aligned}
 v&\geq 0 &&\text{ on $H'\setminus U$;}\\
  v(\bar{x})&=-v(x)&& \text{ for all $x\in \R^{N}$.}
 \end{aligned}
\]
Here $H':=\R^{N}\setminus H$. Thus we have, that $v$ is an entire antisymmetric supersolution
on $U$ with $v\geq 0$ on $H'$ by the weak maximum principle. The strong maximum principle (Corollary  \ref{strong1}) then implies $v\equiv0$ on $\R^{N}$ or $v>0$ in $U$.

  We will show, that $v>0$ in $U$ is not
  possible. This will be separated in to two cases.\\

\noindent \textbf{Case 1)} First assume we are in the first case,
i.e. there is some point $P_{0}\in \overline{\Omega}\cap
\overline{U}\setminus T$.\\
 Note, that we have $P_{0}\in
\partial\Omega\cap \partial U$ due to the choice of $\lambda_0$, we
have $u(P_{0})=0=u(\bar{P_{0}})$, especially $v(P_{0})=0$. Since
$v>0$ in $U$ Hopf's Lemma (Proposition \ref{hopf}) gives, that
$\left(\partial_{\eta}\right)_{s}v(P_{0})<0$, where $\eta$ is the
outernormal at $P_{0}$ on $\partial U$. But since
$\left(\partial_{\eta}\right)_{s}u(P_{0})=c=\left(\partial_{\eta}\right)_{s}u(\bar{P_{0}})$
we must have $\left(\partial_{\eta}\right)_{s}v(P_{0})=0$ which is a
contradiction and thus we cannot be in the first case.\\

\noindent \textbf{Case 2)} Assume  that $v>0$ in $U$ and that $T$ is
orthogonal to
$\partial\Omega$ at a point $P_0\in T\cap\partial\Omega$.\\
Up to translation and rotations, we may assume that $P_0=0$ is
the origin,  $e=e_1$, the interior normal of $\partial\Om$ at the origin is
$e_2$ and $\n^2\delta_\Omega(0)$ is diagonal.
 Without loss of generality, we may also assume  that $\lambda=0$.
\begin{lemma}\label{0-derivative}
We have
\[
v(t\bar{\eta})= o(t^{1+s}),\quad \text{ as $t\to 0$,}
\]
where $\bar{\eta}=(-1,1,0,\dots,0)$.
\end{lemma}
\begin{lemma}\label{lem:max-corner}
Let $\Omega\subset \R^{N}$, $N\geq2$ be an open bounded set with
$C^{2}$ boundary such that  the origin $0\in\partial\Omega$. Assume that the hyperplane  $\{x_1=0\}$ is orthogonal to $\partial
\Omega$ at   $0$. Let
$D\subset\Omega$ be an open set with $C^{2}$ boundary and  symmetric    about $\{x_1=0\}$.
Let $D^{\ast}:=D\cap \{x_{1}<0\}$. Let $c\in L^{\infty}(D^\ast)$ and
$w$ be an antisymmetric supersolution of
\[
 \begin{aligned}
 \hl w &\geq c(x)w && \text{ in $D^{\ast}$;}\\
 w&\geq 0 &&\text{ in $\{x_{1}<0\} $;}\\
 w&>0 && \text{ in $D^{\ast}$.}
 \end{aligned}
\]
 Then letting
$\bar\eta=(-1,1,0\ldots,0)$, there exists $C,t_0>0$ depending only
on $\Omega,N,s$ such that
\[
w(t\bar{\eta})\geq C t^{1+s}\qquad \forall t\in(0,t_0).
\]
\end{lemma}
\noindent
Then applying  Lemma \ref{0-derivative} and \ref{lem:max-corner}
(see the proofs below),  we reach a contradiction. Therefore $v\equiv 0$ as
desired.\\

 If $v\equiv 0$ on $\R^{N}$, then $u\equiv 0$ on $\R^{N}\setminus \tilde{U}$, where $\tilde{U}=U\cup Q_{\lambda_0}(U)\cup (T\cap \Omega)$.  This implies that  $\Omega=\tilde{U}$ yielding  symmetry about $T$. It is then clear that $\tilde{U}$ might have many components  lined up along $T$.
\begin{equation}\label{eq:mantyComp}
\textrm{Assume by contradiction that there are two   connected components  $\tilde{U}_1$ and $\tilde{U}_2$.}
\end{equation}
 \textbf{Observation:} There exists a plane $T'$ perpendicular to $ T$ and  separating  $\tilde{U}_1$ and $\tilde{U}_2$. Otherwise one surround the other (recall that they cannot meet at any boundary points by $C^2$ regularity of $\partial\tilde{U}$) contradicting the minimality of $\lambda_0$.\\

We now move the plane  $T'$ touching, say, $\tilde{U}_1$ first. This leads to  property \eqref{eq:defl0} with some  $\lambda_1$ with  direction $e^1$  and $ T'=T_{\lambda_1}$.
Then the same argument as a above yields
 $$
 \tilde{U}=\tilde{U}_1=\tilde{U}\cup Q_{\lambda_1}(\tilde{U})\cup (T_{\lambda_1}\cap \tilde{U}).
 $$
 By symmetry of $u$ with respect to   $T_{\lambda_1}$ and since $u=0$ in $\R^N\setminus \tilde{U}_1$, we deduce  that  $u$ vanish in $\tilde{U}_2$ which is in contradiction with the fact that $u$ is positive in $\Omega$. Hence $\Omega=\tilde{U}$ is connected with $C^2$ boundary.\\

\noindent
Restarting the moving plane process, we conclude that $\Omega$ is symmetric with respect to all planes for which Situation 1 and/or Situation 2 occur for a first time. We then conclude that $\Omega$ must be a ball.
\end{proof}

We observe that Lemma \ref{0-derivative}  states that "derivatives of order $1+s$"
of $v$ vanish at  the origin.
\begin{proof}[Proof of Lemma \ref{0-derivative}]
Thanks to  \cite[Theorem 1.2]{RS12}, we can write
$$
u(x)=\delta^s(x)\psi(x),
$$
where $\psi(x)\in C^{0,\a}(\ov{\Om})$ for some $\alpha \in (0,1)$ (recall that $\d=\d_\Omega$ is the distance function to $\partial\Omega$). It is clear from our hypothesis that
\begin{equation}
\psi(x)=-c\quad\forall x\in\de\Om.
\end{equation}
Put $\bar{u}(x)=u(\bar{x})=u(-x_1,x_2,\dots,x_N)$,
$\bar\delta(x)=\delta(\bar x)$ and $\bar\psi(x)=\psi(\bar x)$. By continuity, we have
$$
\psi(t\bar{\eta})=-c+o(1)=\bar{\psi}(t\bar{\eta}), \quad\text{ as
$t\to 0$.}
$$
Then, using \eqref{missing assumption}, we have
\begin{equation}\label{eq:u-baru}
	\begin{split}
v(t\bar{\eta})=u(t\bar{\eta})-\bar{u}(t\bar{\eta})&=[\delta^s(t\bar{\eta})-\bar{\delta}^s(t\bar{\eta})]\psi(t\bar{\eta}) +\delta^s(t\eta)[\psi(t\bar{\eta})-\psi(t\eta)]\\
&=[\delta^s(t\bar{\eta})-\bar{\delta}^s(t\bar{\eta})](c+o(1))+o(t^{1+s}),
\quad\text{ as $t\to 0$.}
\end{split}
\end{equation}
By Taylor expansion, we have
$$
\delta(t\bar{\eta})=\delta(0)+\n\delta(0)\cdot(t\bar{\eta})+\frac{1}{2}
\n^2\delta(0)[(t\bar{\eta})]\cdot (t\bar{\eta})+o(t^2), \quad\text{
as $t\to 0$}
$$
and
$$
\bar{\delta}(t\bar{\eta})=\delta(0)+\n\bar\delta(0)\cdot(t\bar{\eta})+\frac{1}{2}
\n^2\bar\delta(0)[(t\bar{\eta})]\cdot (t\bar{\eta})+o(t^2),
\quad\text{ as $t\to 0$.}
$$
 In addition, since  $e_2=\n\d(0)$ is the normal direction, $\partial_{x_i}\d(0)=0$ for all $i\neq
 2$. Therefore
$$
\nabla\delta(0)\cdot\bar\eta=\nabla \bar\delta(0)\cdot\bar
\eta=e_2\cdot\bar\eta=1.
$$
Since $\n^2\delta(0)$ is diagonal, it is plain that
$$
\n^2\delta(0)[\bar{\eta}]\cdot \bar{\eta}=\n^2\bar\delta(0)[\bar{\eta}]\cdot \bar{\eta}=\n^2\delta(0)[e_2]\cdot e_2+\n^2\delta(0)[e_1]\cdot e_1 .
$$
 It follows that
$$
\delta^s(t\bar{\eta})=t^s(1+\frac{s}{2}
\n^2\delta(0)[\bar{\eta}]\cdot (t\bar{\eta})+o(t)), \quad\text{ as
$t\to 0$}
$$
and
$$
\bar{\delta}^s(t\bar{\eta})=t^s(1+\frac{s}{2}
\n^2\delta(0)[\bar{\eta}]\cdot (t\bar{\eta})+o(t)), \quad\text{ as
$t\to 0$.}
$$
We then conclude that
$$
\delta^s(t\bar{\eta})-\bar{\delta}^s(t\bar{\eta})= o(t^{1+s}),
\quad\text{ as $t\to 0$.}
$$
This together with \eqref{eq:u-baru}  proves the claim.
\end{proof}

We also observe that Lemma \ref{lem:max-corner} can be seen  as the Serrin's corner boundary point
lemma.
\begin{proof}[Proof of Lemma \ref{lem:max-corner}]
Let $R>0$ small so that  $B:=B_R(Re_2)\subset \O$ and $\de
B_R(Re_2)\cap \de \O=\{0\} $. Put
$$
K=  B_R(Re_2)\cap \{x_1<0\}.
$$

Define  $B^2= B_R(4R\eta)$ and  $B^1= B_R(4R\bar{\eta})$, where
$\bar{\eta}=e_2-e_1$. From now on we will consider $R$ small such
that $ B^1  \cup B^2\subset \subset D $ (see Figure 2 below):

   \begin{center}%
\begin{figure}[htb]
   \begin{center}%
  \fbox{\begin{minipage}{8cm}%
   \begin{center}%
     \includegraphics[width=\textwidth]{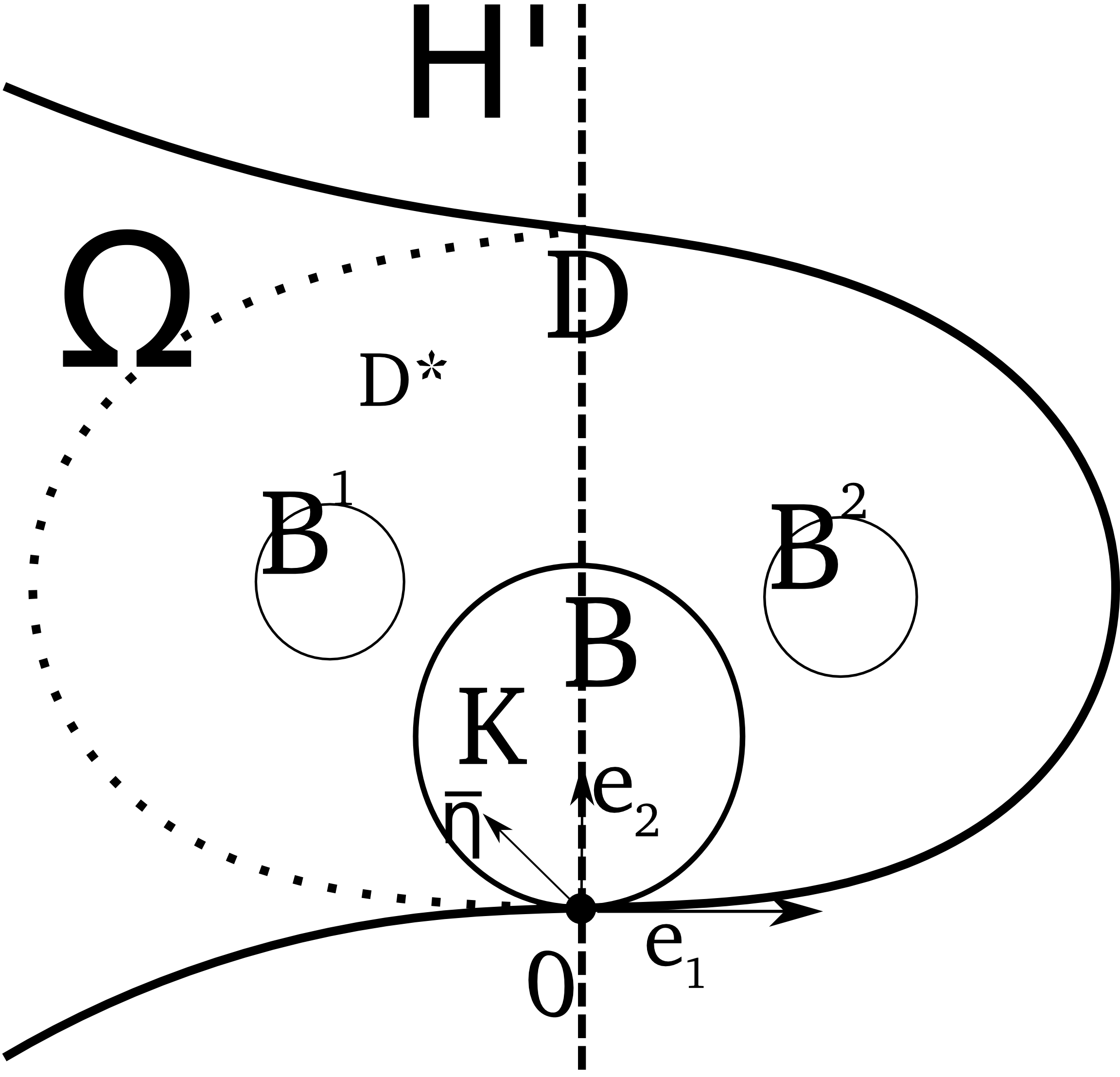}%
   \end{center}%
  \end{minipage}}%
\caption{  $B^2=Q(B^1)$, where $Q$ is the reflection at $\partial H'$}%
   \end{center}%
\end{figure}
   \end{center}%

We next consider the truncated distance
functions to the boundary of these balls denoted by
$$
d^2(x)=(R-|x-4R\eta|)_+,\qquad  d^1(x)=(R-|x-4R\bar{\eta}|)_+.
$$
As in Section \ref{anti-results} we use
$$
\phi_{B}(x)=(R^2-|x-Re_2|^2)_+^s,
$$
and for $\a >0$ (to be chosen later),  we consider the   barrier
$$
h(x)=-x_1[\phi_R(x)+\a (d^1(x)+ d^2(x)) ].
$$
Note that $h(x)=-h(\bar{x})$ and $h\in C^{1,1}( B)\cap \cD^{s}(K)$.
Using  \cite[Theorem 1 + Table 3, pp.549]{D12}, together  with a
scaling and translation,
\begin{equation}\label{eq:Dsx1phiR}
 |\hl (x_1\phi_R(x))|=
|C_{N,s} R^{-1} x_1| \leq C |x_1| \qquad\forall x\in K,
\end{equation}
where here and in the following $C$ is a positive  constant
(possibly depending on $R$, $N$, $s$) but never on $\alpha$. Now we
put
\[
I(x):= -\hl\left[x_{1}(d^1(x)+d^2(x))\right].
\]
Then for  $x\in K$ we have
\begin{align*}
I(x)&=- P.V. \int_{\R^{N}}\frac{-y_{1}(d^1(y)+d^2(y))}{|x-y|^{N+2s}}\ dy\\
&=\int_{B^1}y_{1}d^1(y)\left(|x-y|^{-N-2s}-|x-\bar{y}|^{-N-2s}\right)\ dy\\
&=\int_{B^1}y_{1}d^1(y)|x-y|^{-N-2s}\left(1-\left(\frac{|x-y|}{|x-\bar{y}|}\right)^{N+2s}\right)\
dy,\\
&=\int_{B^1}y_{1}
d^1(y)|x-y|^{-N-2s}\left(1-\left(\frac{|x-y|^2}{|x-y|^2+4x_{1}y_{1}}\right)^{(N+2s)/2}\right)\,dy.
\end{align*}
Observe that by construction,
$$
|x-y|>R \qquad \forall x\in \cap B_R(Re_2),\quad \forall y\in  B^1.
$$
Using this and the fact that  the map
\[
a\mapsto 1-\left(\frac{a}{a+d}\right)^{k}
\]
is strictly monotone decreasing in $a$ for all $d,k>0$,
 we therefore get
\begin{align*}
I(x)& \leq \int_{B^1}y_{1}
d^1(y)|x-y|^{-N-2s}\left(1-\left(\frac{R^2}{R^2+4x_{1}y_{1}}\right)^{(N+2s)/2}\right)\,dy.
\end{align*}
Hence we have (recall that $y_1<0$)
\begin{align*}
I(x)& \leq - C \int_{B^1}
\left(1-\left(\frac{R^2}{R^2+4x_{1}y_{1}}\right)^{(N+2s)/2}\right)\,dy\\
&= - C \int_{B^1}
\left(1-\left(1-\frac{4x_{1}y_{1}/R^2}{1+4x_{1}y_{1}/R^2}\right)^{(N+2s)/2}\right)\,dy.
\end{align*}
Using the elementary inequality
$$
  (1-t)^\b\leq 1-
t\qquad \textrm{ for } t\in(0,1),\,\,\b >1
$$
we get, for all $x\in K$,
\begin{align*}
I(x) & \leq - C
\int_{B^1}  \frac{4x_{1}y_{1}/R^2}{1+4x_{1}y_{1}/R^2} dy \leq - C |x_1|,
\end{align*}
where we have used the fact that $|y_1|$ is bounded away from 0 as
long as $y\in B^1$ and $|x_1|\leq 5R$. Combining this with \eqref{eq:Dsx1phiR}, we
infer
$$
\hl h(x)-c(x) h(x)\leq (C -\a C) |x_1|\qquad \forall x\in K.
$$
Hence we can choose $\alpha$ so that $\hl h-c(x) h\leq0$ in $K$.  Since
also we can find a positive constant  $M>0$ so that $w\geq M h$ in $\overline{B^1}$, we get immediately $ w- M h\geq0 $ in
$\{x_1<0\}\setminus K$. We then deduce from the weak maximum principle
that $w\geq M h$ in $ D^*$. Now since
$$
h(t\bar{\eta})=t^{1+s}(2R-4t^2),
$$
 the proof follows immediately because $t\bar{\eta}\in D^*$ for $t>0$ small.
\end{proof}

\section{Generalization}\label{s:Gen}

\begin{satz}[Theorem \ref{main2} for $N\geq2$]\label{teil2}
Let $c\in\R$ and  $\Omega\subset\R^{N}$, $N\geq 2$, be an open, bounded set with  $C^{2}$ boundary.
   Furthermore, let $f:\R\to\R$ be locally Lipschitz and assume that there is a function $u\in C^s(\overline{\Omega})$, which is nonnegative and nontrivial in $\O$ and fulfills
\begin{equation}
\left\{\begin{aligned}
 \hl u &= f(u) &&\text{ in $\Omega$;}\\
 u&=0  &&\text{ in $\R^{N}\setminus\Omega$;}\\
 \left(\partial_\eta \right)_{s} u&=c  &&\text{ on $\partial \Omega$.}
\end{aligned}\right.
\end{equation}
 Then $\Omega$ is a ball and $u>0$ in $\Omega$.
\end{satz}

\begin{proof}
Let $e\in S^1$ and  consider $\lambda_{0}$ as defined
in \eqref{eq:defl0} and $U:=\Omega'_{\lambda_0}\subset \Omega$ as before.
We define $v_{\lambda_0}(x):=u(x)-u(\bar{x})$ for all $x\in\R^N$, where we use the notation as usual, i.e. $Q_{\lambda,e}(x)=:\bar{x}$ and $Q_{\lambda,e}$ is the reflection of $T=T_{\lambda,e}$. Then $v_{\lambda_0}$ solves
\[
 \hl v_{\lambda_0}\geq -c_f(x) v_{\lambda_0} \qquad\text{in $U$,}
\]
 where
\[
 c_f(x):=\left\{\begin{aligned}
                 \frac{f(u(x))-f(u(\bar{x}))}{u(x)-u(\bar{x})}, &&\text{ if $u(x)\neq u(\bar{x})$;}\\
0,&&\text{ if $u(x)=u(\bar{x})$.}
                \end{aligned}
 \right.
\]
Let $L_f$ be the Lipschitz constant of $f$ for $\cB=[0,\|u\|_{L^{\infty}(\R^{N})}]$. Then we have $\|c\|_{L^{\infty}(U)}\leq L_f$.   Here, we cannot directly apply the maximum principle to get $v_{\lambda_0}\geq 0$ in $H'$ as in the previous section because $L_f$ might be large. However, by  using  the moving plane method, we can prove that
\begin{equation}\label{eq:vgeq0Hp}
  v_{\lambda_0}\geq 0 \qquad \text{ on $H'$;}
\end{equation}
To this end, we observe that for $\lambda\in (\lambda_0,l)$ but close to $l$   we have $L_f\leq
\lambda_1(\Omega'_{\lambda})  $ so that $u(x)-u(Q_{\lambda,e}(x))\geq 0$
in $\Omega'_{\lambda}$ by the weak  maximum principle. Now by the strong maximum principle
$$
\textrm{ $(S_{\lambda})$ \quad  $v_{\lambda}(x):=u(x)-u(Q_{\lambda}(x))>0\quad$ for all $x\in \Omega'_{\lambda}$}
$$
as $u$ is nontrivial.
We let
\[
 \tilde{\lambda}:=\inf\{\lambda>\lambda_{0}\;:\; (S_{\mu}) \text{ holds for all $\lambda>\mu$}\}.
\]
Our aim is to prove that $\tilde{\lambda}=\lambda_{0}$. Assume by contradiction that  $\tilde{\lambda}>\lambda_{0}$. Then by continuity and the strong maximum principle we have that $(S_{\tilde{\lambda}})$ holds. Since $\tilde{\lambda}>\lambda_{0}$, there is by continuity $\epsilon>0$ such that $\Omega'_{\tilde{\lambda}-\epsilon}\subset\Omega$. Choose an open set $\Pi\subset \Omega'_{\tilde{\lambda}-\epsilon}$ such that $\{v\leq 0\}\cap  \Omega'_{\tilde{\lambda}-\epsilon}\subset \Pi$ and we may assume that  $|\Pi|$ is small by making $\epsilon$ possibly smaller. The maximum principle then can be applied to $\Pi$ giving $v_{\tilde{\lambda}-\epsilon}>0$ in $\Pi$ (as before) and thus $S_{\tilde{\lambda}-\epsilon}$ holds in contradiction to the choice of $\tilde{\lambda}$. Thus $\tilde{\lambda}=\lambda_0$. Hence \eqref{eq:vgeq0Hp} is proved.
We have now that $v$ is an entire antisymmetric supersolution
on $U$ with $v_{\lambda_0}\geq 0$ in $H'$ by the weak maximum principle.
Arguing as in the proof of  Theorem \ref{main} in the previous section, we obtain $v_{\lambda_0}\equiv 0$ in $\R^{N}$. This  implies that $u$ is symmetric with respect to all planes $T_{\lambda_0,e}=T_{\lambda_0(e),e}$ for which \eqref{ST12} occur for a first time. More than that, the moving plane process above yield  monotonicity through lines perpendicular to $T_{\lambda_0,e}$: for every   $e\in S^1$
$$
\textrm{ $ u(x)-u(Q_{\lambda,e}(x))>0\quad$ for all $\lambda\in (\lambda_0,l) $ and  for all $x\in \Omega'_{\lambda,e}$ }.
$$
In particular for
all $e\in S^1$ and for all $\lambda\in \R$, we have either $u(x)\geq u(Q_{\lambda,e}(x))$ for all $x\in H_{\lambda,e}$ or $u(x)\leq u(Q_{\lambda,e}(x))$ for all $x\in H_{\lambda,e}$. Now by a well known result (which we include a proof in the Appendix for the readers convenience), $u$ coincides to a radial   function up to a translation which is decreasing. Hence $\supp(u)$ is a ball.\\
 We claim that $\overline{\Omega}=\supp(u)$. Indeed, assume on the contrary that   $\supp(u)\neq \overline{\Omega}$. Then there is a
 ball $B\subset\subset\Omega\setminus \supp(u)$,  such that $u\equiv 0$ in $B$. Consider the hyperplane $T$ separating $B$ and $\supp(u)$.   It is clear that $u\equiv0$ on the halfspace $H$ with boundary $T$ containing $B$. Let $e\in S^1$ normal to $T$ and contained in $H$.  Now by moving the planes      $T_{\lambda,e}$ as above, we get, for very $\lambda\in (\lambda_0,l)$
 $$
\textrm{ $ u(x) > u(Q_{\lambda,e}(x))\geq0\quad$    for all $x\in \Omega'_{\lambda,e}$ }.
$$
This, in particular, implies that $u>0$ in $ \Omega\cap H_{\lambda_0,e} $ which is impossible.

\end{proof}

\begin{satz}[Theorem \ref{main} and Theorem \ref{main2} for $N=1$]\label{N=1}
 Let $c\in\R$ and  $\Omega\subset\R$ be a bounded open set. Let $f:\R\to\R$ be locally Lipschitz and assume that there is a function $u\in C(\overline{\Omega})$, which is nonnegative, nontrivial in $\O$ and satisfies
\begin{equation}
\left\{\begin{aligned}
 \hl u &= f(u) &&\text{ in $\Omega$;}\\
 u&=0  &&\text{ in $\R\setminus\Omega$;}\\
 \left(\partial_\eta \right)_{s} u&=c  &&\text{ on $\partial \Omega$.}
\end{aligned}\right.
\end{equation}
 Then $\Omega=(\a,\b)$ for some $\a,\b\in \R$, $\a<\b$ and $u>0$ in $\Omega$.
\end{satz}
\begin{proof}
Now assume that $\Omega$ has at least two different connected components $(\a,\b)$ and $(a,b)$  with $a<b<\a<\b$.  Note that as in the case $N\geq 2$ we  can move points (instead of moving planes!) from the right up to $\lambda_{0}=(\a+\b)/2$, so that $v(x):=u(x)-u(\bar{x})$ solves
\[
 \hl v\geq -c_f(x) v \qquad\text{in $(a,\lambda_0)$}
\]
and $v(x)\geq 0$ for $x<\lambda_0 $ by arguing as in the proof of Theorem \ref{teil2}. Note that only interior touching can occur.
 Hence by Hopf's Lemma we obtain $v\equiv 0$ on $\R$, but this gives $u\equiv 0$ on $\R\setminus (\a,\b)$. Next moving from the left up to  $\lambda_{0}=(a+b)/2$ implies, as previously,    $u\equiv 0$ in $(a,b)$. Therefore $u\equiv0$ in $\R$ leading to a contradiction. The positivity of $u$ finally follows as in Theorem \ref{teil2} by the monotonicity which is a byproduct of the  moving plane method.
\end{proof}

\section{Appendix}
The result below was stated in \cite{T}.
\begin{prop}\label{Tobias}
Let $u:\R^{N}\to\R$ be continuous and such that $\lim_{|x|\to\infty}u(x)=c_\infty\in \R\cup\{\pm\infty\}$ exists. Then the following statements are equivalent:
\begin{enumerate}
\item[(i)] There is $z\in \R^{N}$, such that $u(\cdot-z)$ or $-u(\cdot-z)$ is radially symmetric and decreasing in the radial variable.
\item[(ii)] For every half space $H\in \cH$ we either have $u(x)-u(Q_H(x))\geq0$ for all $x\in H$ or $u(x)-u(Q_H(x))\leq0$ for all $x\in H$. We say $H$ is dominant or subordinate for $u$ respectively.
\end{enumerate}
Here $\cH$ is the set of all affine half spaces in $\R^{N}$ and $Q_{H}$ is the reflection at $\partial H$ for any $H\in \cH$.
\end{prop}
The following proof is communicated to the authors by Tobias Weth.
\begin{proof}
Assertion $(i)$ implies $(ii)$ is obvious and can be found in  \cite{T}. Thus we only need to show that $(ii)$ implies $(i)$. Without restriction, we may assume that $u$ is not constant. Consider the halfspace $H_\lambda=\{x\in \R^{N}\;:\; x_1>\lambda\}$, for $\lambda\in \R$ and the open set
\[
I:=\{\lambda\in \R\;:\; u(x)>u(Q_{H_\lambda}(x)) \text{ for some $x\in H_{\lambda}$}\}.
\]
Since $u$ is not constant, we may assume, replacing $u$ with $-u$ if necessary, that $u$ is larger than $c_\infty$ at some point in $\R^{N}$. Hence there exists a maximal $s_1\in \R\cup \{\infty\}$ such that $I$ contains the interval $(-\infty,s_1)$. By $(ii)$ the half spaces $H_\lambda$, $\lambda\in(-\infty,s_1)$ are dominant for $u$, which implies that
\begin{equation}\label{weth}
\textrm{ $u$ is nondecreasing in $x_1$ in the set $\{x\in \R^{N}\;:\; x_1<s_1\}$.}
\end{equation}
This forces $s_1<\infty$, since $u$ is not constant and tends to $c_\infty$ as $|x|\to\infty$. By the maximal choice of $s_1$ and the continuity of $u$, $u$ is symmetric with respect to the hyperplane $\{x_1=s_1\}$. By the same argument, we can find $s_i\in \R$, such that $u$ is symmetric with respect to the hyperplanes $\{x_i=s_i\}$ for $i=2,\ldots,N$. Translating $u$ if necessary, we may assume, that $s_i=0$ for $i=1,\ldots,N$, so that $u$ is symmetric with respect to all coordinate reflections. This implies that $u$ is even, i.e. $u(x)=u(-x)$ for all $x\in \R^N$. As a consequence, if $H\in \cH_0=\{H\in \cH\;:\; 0\in \partial H\}$ is such that $u(x)\geq u(Q_{H}(x))$ for all $x\in H$, then also $u(Q_{H}(x))=u(-Q_{H}(x))\geq u(-x)=u(x)$ for every $x\in H$, so that $u$ is symmetric with respect to the hyperplane $\partial H$. This implies that $u$ is symmetric with respect to any hyperplane containing $0$, so that $u$ is radially symmetric (see e.g. \cite[Section 2.1]{T}). Finally, (\ref{weth}) implies that $u$ is decreasing in the radial variable.
\end{proof}

\bibliographystyle{amsplain}

\begin{thebibliography}{10}


\bibitem{AM}
V. Agostiniani and R. Magnanini.
\emph{ Symmetries in an overdetermined problem for the {G}reen's function.}
 { Discrete Contin. Dyn. Syst. Ser. S} \textbf{4.4} (2011), 791--800.


\bibitem{Alexandrov} A.~D.~Alexandrov, \emph{Uniqueness theorems for surfaces in the large I}, Vestnik Leningrad Univ. Math.
\textbf{11} (1956), 5--17.

\bibitem{BLW05}
M.~Birkner, J.~A.~L\'opez-Mimbela, and A.~Wakolbinger, \emph{Comparison results
  and steady states for the {F}ujita equation with fractional {L}aplacian},
  Annales de L'Institut Henri Poincar\'e \textbf{22} (2005), 83--97.

\bibitem{BD}
I.~Birindelli and F.~Demengel.
\newblock \emph{Overdetermined problems for some fully non linear operators},
\newblock { Comm. Part. Diff. Eq.} \textbf{38.4} (2013), 608--628.


\bibitem{BB00} K.~Bogdan and T.~Byczkowski, \emph{Potential Theory of Schr\"odinger Operator based on fractional Laplacian},
Probability and Mathematical Statistics,  \textbf{2.20} (2000), 293--335.

\bibitem{BNST} B.~Brandolini, C.~Nitsch, P.~Salani and C.~Trombetti, \emph{Serrin-type overdetermined problems: an alternative proof}, Arch. Ration. Mech. Anal. \textbf{190.2} (2008), 267--280.

\bibitem{BH}
F.~Brock and A.~Henrot,
\newblock \emph{A symmetry result for an overdetermined elliptic problem using
  continuous rearrangement and domain derivative},
\newblock { Rend. Circ. Mat. Palermo (2)} \textbf{51.3} (2002), 375--390.

\bibitem{BK}
G.~Buttazzo and B.~Kawohl,
\newblock \emph{ Overdetermined boundary value problems for the {$\infty$}-laplacian},
\newblock { International Mathematics Research Notices 2011}  \textbf{2}, 237--247.


\bibitem{Chen_Li_Ou} W. Chen, C.  Li and Biao Ou: {\sl Classification of solutions
for an integral equation.}  Comm. Pure Appl. Math. {\bf 59}  (2006), 330--343.



\bibitem{CH}
M. Choulli and A. Henrot,
\newblock  \emph{Use of the domain derivative to prove symmetry results in partial
  differential equations},
\newblock { Math. Nachr.} \textbf{192} (1998), 91--103.

\bibitem{CS}
A.~Cianchi and P.~Salani.
\newblock  \emph{Overdetermined anisotropic elliptic problems},
\newblock { Math. Ann.}  \textbf{345.4} (2009), 859--881.

\bibitem{dLS}
F. Da~Lio and B. Sirakov.
\newblock \emph{ Symmetry results for viscosity solutions of fully nonlinear uniformly
  elliptic equations},
\newblock { J. Eur. Math. Soc.} \textbf{9.2} (2007), 317--330.

\bibitem{DGV12}
 A.-L.~Dalibard and D.~G\'erard-Varet, \emph{On shape optimization problems involving the fractional laplacian},
ESAIM. Control, Optimisation and Calculus of Variations \textbf{19.4} (2013), 976--1013.


\bibitem{NPV11}
E.~di~Nezza, G.~Palatucci, and E.~Valdinoci, \emph{Hitchhiker's {G}uide to the
  {F}ractional {S}obolev {S}paces}, Bulletin des Sciences Math\'e matiques
   \textbf{136} (2012), 521--573.
   
\bibitem{DPTV23} S. Dipierro, G. Poggesi, J. Thompson, E. Valdinoci, {\em Quantitative stability for the nonlocal overdetermined Serrin problem}, preprint available on arXiv, 2023, arXiv:2309.17119v1.


\bibitem{D12} B.~Dyda, \emph{Fractional calculus for power functions and eigenvalues of the fractional Laplacian},
Fractional Calculus and Applied Analysis \textbf{15.4} (2012), 536--555.

\bibitem{EG92} L.~C.~Evans and R.~F.~Gariepy, \emph{Measure Theory and Fine Properties of Functions}, CRC Press, Boca Raton, 1992.

\bibitem{FW2}M. M. Fall, T. Weth, \emph{Monotonicity and nonexistence results for some fractional elliptic problems in the half space}, preprint (2013) available online at \url{http://arxiv.org/abs/1309.7230}.


\bibitem{FK}
A.~Farina and B.~Kawohl,
\newblock \emph{Remarks on an overdetermined boundary value problem},
\newblock { Calc. Var. Partial Differential Equations} \textbf{31.3} (2008), 351--357.

\bibitem{FV1}
A. Farina and E. Valdinoci,
\newblock \emph{Flattening results for elliptic {PDE}s in unbounded domains with
  applications to overdetermined problems},
\newblock { Arch. Ration. Mech. Anal.}  \textbf{195.3} (2010), 1025--1058.

\bibitem{FV2}
A. Farina and E. Valdinoci,
\newblock \emph{Overdetermined problems in unbounded domains with {L}ipschitz
  singularities},
\newblock { Rev. Mat. Iberoam.} \textbf{26.3} (2010), 965--974.

\bibitem{FQT} P. Felmer, A. Quaas, J. Tan
  \emph{Positive solutions of Nonlinear Schr\"{o}dinger equation with the fractional Laplacian}, {to appear in Proc. Roy. Soc. Edinburgh.}



\bibitem{FG}
I. Fragal{\`a} and F. Gazzola,
\newblock \emph{ Partially overdetermined elliptic boundary value problems},
\newblock { J. Differential Equations} \textbf{245.5} (2009), 1299--1322.

\bibitem{FGK}
I. Fragal{\`a}, F. Gazzola, and B. Kawohl,
\newblock \emph{Overdetermined problems with possibly degenerate ellipticity, a
  geometric approach},
\newblock { Math. Z.} \textbf{254.1} (2006), 117--132.

\bibitem{FW13}
P.~Felmer and Y.~Wang
\newblock \emph{Radial symmetry of positive solutions involving the fractional Laplacian},
\newblock {Commun. Contemp. Math} (2013), available online at \url{http://www.worldscientific.com/doi/pdf/10.1142/S0219199713500235}.

\bibitem{GL}
N.~Garofalo and J.L. Lewis,
\newblock  \emph{A symmetry result related to some overdetermined boundary value
  problems},
\newblock { American Journal of Mathematics} \textbf{111.1} (1989), 9--33.


 \bibitem{GNN1} B.~Gidas, W.-M.~Ni and L.~Nirenberg, \emph{Symmetry and related problems via the maximum principle}, Comm. Math. Phys.  \textbf{68}  (1979),  209--243.

 \bibitem{GNN2} B.~Gidas, W.-M.~Ni and L.~Nirenberg, \emph{Symmetry of positive solutions of nonlinear equations}, Math. Anal. Appl. Part A, Adv. Math. Suppl. Studies A   \textbf{7}  (1981), 369--402.


\bibitem{HHP}
L.  Hauswirth, F. H{\'e}lein, and F. Pacard, \emph{On an overdetermined elliptic problem}, Pacific J. Math. \textbf{250.2} (2011), 319--334.

\bibitem{JW} S. Jarohs, T. Weth,  \emph{Asymptotic symmetry for a class of fractional reaction-diffusion equations},
          Discrete Contin. Dyn. Syst. \textbf{34.6} (2014), 2581--2615.

\bibitem{JKS25} S. Jarohs, T. Kulczycki, and P. Salani, \emph{Qualitative properties of free boundaries for the exterior Bernoulli problem for the half Laplacian}, J. Math. Anal. \textbf{547.1} (2025), Article ID 129285, pp. 18.

\bibitem{PSV13} G.~Palatucci, O.~Savin and E.~Valdinoci, \emph{Local and global minimizers for a variational energy involving a fractional norm}, Ann. Mat. Pura Appl. (4) \textbf{192.4} (2013), 673--718

\bibitem{Pr}
J.~Prajapat, \emph{ Serrin's result for domains with a corner or cusp}, Duke mathematical journal, \textbf{91.1} (1998), 29--31.


\bibitem{Ra}
A.~G. Ramm,
\newblock \emph{ Symmetry problem},
\newblock { Proc. Amer. Math. Soc.}  \textbf{141.2} (2013), 515--521.

\bibitem{Re2}
W.~Reichel, \emph{Radial symmetry for an electrostatic, a capillarity and some fully
  nonlinear overdetermined problems on exterior domains}, Z. Anal. Anwendungen \textbf{15.3} (1996), 619--635.

\bibitem{Re1}
W  Reichel,
\newblock \emph{Radial symmetry for elliptic boundary-value problems on exterior
  domains},
\newblock { Arch. Rational Mech. Anal.} \textbf{137.4} (1997), 381--394.


\bibitem{RS12}
       X.~Ros-Oton and J.~Serra, \emph{The Dirichlet Problem for the fractional Laplacian: Regularity up to the boundary},
        J. Math. Pures Appl. 101 (2014), 275--302.


\bibitem{S71}
J.~Serrin, \emph{A Symmetry Problem in Potential Theory},
  Arch. Rational Mech. Anal. \textbf{43} (1971), 304--318.





\bibitem{S05} L.~Silvestre, \emph{Regularity of the obstacle problem for a fractional power of the
   Laplace operator}, Comm. Pure Appl. Math. \textbf{60.1} (2007), 67--112.

\bibitem{SS}
L.~Silvestre and B.~Sirakov,
\newblock  \emph{Overdetermined problems for fully for fully nonlinear
  elliptic equations}, preprint (2013) available online at \url{http://arxiv.org/abs/1306.6673}.

\bibitem{Si}
B. Sirakov,
\newblock \emph{Symmetry for exterior elliptic problems and two conjectures in
  potential theory},
\newblock { Ann. Inst. H. Poincar\'e Anal. Non Lin\'eaire} \textbf{18.2} (2001), 135--156.




 \bibitem{Wein-Serrin} H.~F.~Weinberger, \emph{Remark on the preceding paper of Serrin},
Arch. Rational Mech. Anal. \textbf{43} (1971), 319--320.

\bibitem{T} T.~Weth, \emph{Symmetry of solutions to variational problems for nonlinear elliptic equations via reflection methods}, Jahresber. Deutsch. Math.-Ver. \textbf{112} (2010), 119--158.

\bibitem{YY} S.~Y.~Yolcu, T.~Yolcu,  \emph{Estimates for the sums of eigenvalues of the fractional Laplacian on a bounded domain}, Commun. Contemp. Math. \textbf{15.3} (2013), 1250048.











































\end{thebibliography}

\end{document}